\documentclass[12pt, oneside]{article}   	
\usepackage{geometry}                		
\usepackage{amsthm}
\usepackage{amsmath}
\usepackage{mathtools}
\usepackage{graphicx}				
							 	
\usepackage{amssymb}
\usepackage[all]{xy}
\usepackage{makeidx}
\usepackage{nameref}
\usepackage{xcolor}
\usepackage{tikz-cd}
\usetikzlibrary{arrows,decorations.pathreplacing,angles,quotes,bending}
\usepackage{mathscinet}
\usepackage{enumerate}
\usepackage{float}
\usepackage{enumitem}
\usepackage[font=scriptsize]{caption}
\makeindex
\theoremstyle{plain}
\makeatletter
\newcommand{\newreptheorem}[2]{\newtheorem*{rep@#1}{\rep@title}\newenvironment{rep#1}[1]{\def\rep@title{#2 \ref*{##1}}\begin{rep@#1}}{\end{rep@#1}}}
\makeatother

\newtheorem{theorem}{Theorem}[section]
\newreptheorem{theorem}{Theorem}

\newtheorem{lemma}[theorem]{Lemma}
\newtheorem{corollary}[theorem]{Corollary}
\newtheorem{proposition}[theorem]{Proposition}
\theoremstyle{definition}
\newtheorem{definition}[theorem]{Definition}

\title{The braid group $B_3$ in the framework of continued fractions}
\author{Amitesh Datta}
\begin{document}
\maketitle

\abstract{We use the classical interpretation of the braid group $B_3$ as a central extension of the modular group $\text{PSL}_2\left(\mathbb{Z}\right)$ to establish new and fundamental properties of $B_3$ using the theory of continued fractions. In particular, we give simple and natural linear time algorithms to solve the word and conjugacy problems in $B_3$. The algorithms introduced in this paper are easy to implement and are the most efficient algorithms in the literature to solve these problems in the braid group $B_3$.}
\tableofcontents
\section{Introduction}
\subsection{Background and motivation}
The braid groups $B_n$ are classical and important groups in mathematics, explicitly introduced by Artin~\cite{artinbraid} in the 1920s. We briefly review Artin's definition of $B_n$ as the group of geometric $n$-braids. Let $\mathbb{I}$ denote the closed unit interval. A geometric $n$-braid is the isotopy class relative to its boundary of a (one-dimensional) cobordism in $\mathbb{R}^2\times \mathbb{I}$ between $\{\left(1,0\right),\dots,\left(n,0\right)\}$ and $\{\left(1,1\right),\dots,\left(n,1\right)\}$, with the property that the intersection of the cobordism with each cross-section $\mathbb{R}^2\times \{t\}$ for $t\in \mathbb{I}$ consists of $n$ points. A natural group operation of concatenation of cobordisms followed by rescaling can be defined on geometric $n$-braids, and this results in the braid group $B_n$.

The braid groups can also be defined algebraically via the (finite) Artin presentation: \[B_n = \left\langle  \sigma_1,\dots,\sigma_{n-1}:\sigma_i\sigma_{i+1}\sigma_i = \sigma_{i+1}\sigma_i\sigma_{i+1}, \sigma_i\sigma_j = \sigma_j\sigma_i\right\rangle,\] where $1\leq i\leq n-2$ in the first set of relations and $\left|i-j\right|>1$ in the second set of relations. The \textit{word problem} in a finitely generated group is to find an algorithm that decides if a pair of words in the generators represent the same group element. The \textit{conjugacy problem} in a finitely generated group is to find an algorithm that decides if a pair of words in the generators represent conjugate group elements.

The braid groups are fundamental in the theory of knots and links in three-dimensional space. If $g\in B_n$, then the braid closure of $g$ is the isotopy class of the link obtained by connecting the endpoints $\left(i,0\right)$ for $i\in \{1,\dots,n\}$ of $g$ in $\mathbb{R}^2\times \{0\}$ to the corresponding endpoints $\left(i,1\right)$ for $i\in \{1,\dots,n\}$ of $g$ in $\mathbb{R}^2\times \{1\}$ by $n$ unknotted, parallel, simple arcs. Alexander's theorem and Markov's theorem assert that the isotopy class of every link in three-dimensional space is a braid closure, and a pair of braids have the same braid closures if and only if they are related by a finite sequence of conjugations, stabilizations (a stabilization is a map $B_n\to B_{n+1}$ defined by $\sigma\to \sigma\sigma_n^{\pm 1}$) and destabilizations (a destabilization is the inverse of a stabilization, when the inverse exists). In particular, the study of the word and conjugacy problems in the braid groups is an important step toward a complete understanding of knots and links in three-dimensional space.

In this work, we focus on the braid group $B_3$, which already exhibits a rich structure. The modular group \[\text{PSL}_2\left(\mathbb{Z}\right) = \text{SL}_2\left(\mathbb{Z}\right)/\left\langle - I \right\rangle.\] is naturally intertwined with the theory of continued fractions and there is a classical short exact sequence \[1\to \left\langle \Delta^2 \right\rangle \to B_3\to \text{PSL}_2\left(\mathbb{Z}\right)\to 1,\]  where $\Delta^2$ is a generator of the center of $B_3$ ($\Delta\in B_3$ is a distinguished element known as the Garside element of $B_3$). The short exact sequence exhibits $B_3$ as a central extension of $\text{PSL}_2\left(\mathbb{Z}\right)$, and we exploit this interpretation of $B_3$ to establish new and fundamental group-theoretic properties of $B_3$. In particular, we establish simple and natural linear time algorithms to solve the word and conjugacy problems in $B_3$, in terms of continued fractions. The algorithms are easy to implement, and are the most efficient algorithms to solve these problems for $B_3$ in the literature.

We briefly summarize the history of the word and conjugacy problems in the braid groups. The original solution to the word problem in the braid groups is due to Artin, and has exponential running time. The algorithm is based on the interpretation of the braid group $B_n$ as the mapping class group of the $n$-times punctured disk. In his seminal work in the 1960s, Garside~\cite{garside1969braid} established exponential time algorithms to solve the word and conjugacy problems in the braid groups. In the 1980s, Thurston (Chapter 9 of~\cite{thurstonbraid}) established a quadratic time algorithm to solve the word problem in the braid groups, and also an exponential time algorithm to solve the conjugacy problems in the braid groups. Birman-Ko-Lee~\cite{birmankolee} and Franco-Gonz{\'a}lez-Meneses~\cite{francobraidconjugacy} have also established exponential time algorithms to solve the conjugacy problem in the braid groups, and these are currently the most efficient algorithms to solve the conjugacy problem in the braid groups. Finally, we remark that Xu~\cite{xubraid} has studied the word and conjugacy problems specifically in $B_3$.

\subsection{Summary of the main results}
We summarize the main results in this paper. Firstly, we set up the notation and terminology that are necessary to formally state the main results. We will define a set of numerical invariants for braids in $B_3$ that we will use to solve the word and conjugacy problems. The numerical invariants are derived from a map $B_3\to \text{PSL}_2\left(\mathbb{Z}\right)$ (see Subsection~\ref{subsecfses} of Section~\ref{seccf} for a definition of this map).

We consider the Artin presentation \[B_3 = \left\langle \sigma_1,\sigma_2 : \sigma_1\sigma_2\sigma_1 = \sigma_2\sigma_1\sigma_2 \right\rangle.\] Let $g=\prod_{i=1}^{n} \sigma_1^{a_i}\sigma_2^{b_i}\in B_3$ with $a_i,b_i\in \mathbb{Z}$ for $1\leq i\leq n$. Let $\epsilon:B_3\to \mathbb{Z}$ be defined by $\epsilon\left(g\right) = \sum_{i=1}^{n} \left(a_i+b_i\right)$. Let $\mathbb{P}^1$ be the projective line over $\mathbb{Q}$, which can be identified with the set of extended rational numbers $\mathbb{Q}\cup \{\infty\}$. If $c_i\in \mathbb{Z}$ for $0\leq i\leq k$, then \[\left[c_0,c_1,\dots,c_k\right] =  c_0 + \cfrac{1}{c_1 + \cfrac{1}{c_2 + \cfrac{1}{ \ddots + \cfrac{1}{c_n} }}}\] is an extended rational number. Let $\rho_1\left(g\right) = \left[a_1,-b_1,\dots,a_n,-b_n\right]$ and $\rho_2\left(g\right) = \left[a_1,-b_1,\dots,a_n\right]$. We define the map $\rho:B_3\to \left(\mathbb{P}^1\right)^2\times \mathbb{Z}$ by the rule $\rho\left(g\right) = \left(\rho_1\left(g\right),\rho_2\left(g\right),\epsilon\left(g\right)\right)$. 

The map $\rho$ defines a complete set of numerical invariants for braids in $B_3$ and this results in an efficient and natural linear time algorithm to solve the word problem in the braid group $B_3$.

\begin{theorem}
\label{mthmwp}
The map $\rho:B_3\to \left(\mathbb{P}^1\right)^2\times \mathbb{Z}$ is injective. If $g\in B_3$, then $\phi\left(g\right)$ can be computed in running time $\mathcal{O}\left(L\right)$ if the word length of $g$ is $L$. In particular, the word problem in the braid group $B_3$ can be solved in linear time. 
\end{theorem}

We also establish a new normal form for braids in $B_3$. The normal form is a variation of the Garside normal form in $B_3$~\cite{garside1969braid}. Let $\Delta = \sigma_1\sigma_2\sigma_1$ be the Garside element of $B_3$.

\begin{theorem}
\label{mthmnf}
If $g\in B_3$, then we can uniquely write $g = \Delta^{k}\prod_{i=1}^{n} \sigma_1^{a_i}\sigma_2^{b_i}$, where $a_1\geq 0$, $b_n\in \mathbb{Z}$, $a_i> 0$ for $2\leq i\leq n$, and $b_i<0$ for $1\leq i\leq n-1$.
\end{theorem}
 
We now define a map $\mu:B_3\to \mathbb{C}\left(x\right)$ to solve the conjugacy problem in the braid group $B_3$, where $\mathbb{C}\left(x\right)$ denotes the field of rational functions over $\mathbb{C}$ in one variable. Let $g\in B_3$. We define $\rho_{j}^{i}\left(g\right)\in \mathbb{Z}$ to be the numerator of $\rho_j\left(g\right)$ if $i=1$ and the denominator of $\rho_j\left(g\right)$ if $i=2$. (The convention is that $\infty = \frac{1}{0}$.) We define $\text{tr}\left(g\right) = \left|\rho_1^{1}\left(g\right) + \rho_2^{2}\left(g\right)\right|\in \mathbb{Z}$.

If $\text{tr}\left(g\right) = 2$, then we define $s\left(g\right)\in \mathbb{Z}$ as follows. If $\rho_{1}^{1}\left(g\right) = 1$ and $\rho_{1}^{2}\left(g\right) = 0$, then $s\left(g\right) = \rho_{2}^{1}\left(g\right)$. Otherwise, let $u\left(g\right)$ be the greatest common divisor of $\rho_{1}^{1}\left(g\right)-1$ and $\rho_{1}^{2}\left(g\right)$, and choose $r$ and $t$ such that $\left(\rho_{1}^{1}\left(g\right)-1\right)t - \rho_{1}^{2}\left(g\right)r = u\left(g\right)$. Let $s\left(g\right) = u\left(g\right)\left(r + \frac{\rho_{2}^{1}\left(g\right)}{\rho_{1}^{1}\left(g\right)-1}t\right)$ if $\rho_{1}^{1}\left(g\right)\neq 1$, and let $s\left(g\right) = u\left(g\right)\left(r + \frac{\rho_{2}^{2}\left(g\right)-1}{\rho_{1}^{2}\left(g\right)}t\right)$ if $\rho_{1}^{2}\left(g\right)\neq 0$.

If $\text{tr}\left(g\right)>2$, then we define $\kappa\left(g\right)\in \mathbb{R}$ to be the solution to the equation $\rho_1^{2}\left(g\right)x^2 - \left(\rho^{2}_{2}\left(g\right)-\rho_{1}^{1}\left(g\right)\right)x - \rho_2^{1}\left(g\right) = 0$ with $\rho_{1}^{2}\left(g\right)\kappa\left(g\right) + \rho_{2}^{2}\left(g\right)>1$. The real quadratic irrational $\kappa\left(g\right)$ has an (infinite) eventually periodic continued fraction expansion (where the partial quotients are positive integers). The fundamental periods of $\kappa\left(g\right)$ are cyclic shifts of each other, and we can lexicographically order them. Let $\chi\left(g\right) = \left[\overline{c_0,\dots,c_k},\dots\right]$ be the (infinite) purely periodic continued fraction with fundamental period $\left[c_0,\dots,c_k\right]$, where $\left[c_0,\dots,c_k\right]$ is the lexicographically minimal fundamental period of $\kappa\left(g\right)$.

Let $i=\sqrt{-1}$ denote the imaginary unit and let $\omega = e^{2\pi i/3}$ be the cubic root of unity with positive imaginary part. We define $\lambda\left(g\right)$ as follows: \[\lambda\left(g\right) = \begin{cases} 
      i & \text{if } \text{tr}\left(g\right) = 0\\
     \omega &  \text{if } \text{tr}\left(g\right) = 1\text{ and } \rho_{2}^{1}\left(g\right)>0\\
     -\omega & \text{if } \text{tr}\left(g\right) = 1\text{ and } \rho_{2}^{1}\left(g\right)<0\\
     s\left(g\right) & \text{if } \text{tr}\left(g\right) = 2 \\
     \chi\left(g\right) & \text{if } \text{tr}\left(g\right) > 2
   \end{cases}
\] We define $\mu\left(g\right) = \lambda\left(g\right)x^{\epsilon\left(g\right)}$, where $x$ is an indeterminant. The map $\mu:B_3\to \mathbb{C}\left(x\right)$ defines a complete set of class invariants for braids in $B_3$ and this results in a simple linear time algorithm to solve the conjugacy problem in the braid group $B_3$.

\begin{theorem}
\label{mthmc}
The map $\mu:B_3\to \mathbb{C}\left(x\right)$ is a class function on $B_3$. Furthermore, $g,h\in B_3$ are conjugate if and only if $\mu\left(g\right) = \mu\left(h\right)$. If $g\in B_3$, then $\mu\left(g\right)$ can be computed in running time $\mathcal{O}\left(L\right)$ if the word length of $g$ is $L$. In particular, the conjugacy problem in the braid group $B_3$ can be solved in linear time.
\end{theorem}

We also give a new and simpler proof of the following known statement~\cite{allenbyconsep}. A group $G$ is \textit{conjugacy separable} if $g,h\in G$ are conjugate in $G$ whenever $g$ and $h$ are conjugate in all finite quotients of $G$.

\begin{theorem}
\label{mthmcs}
The braid group $B_3$ is conjugacy separable.
\end{theorem}

\subsection{Outline of the paper}
The outline of the paper is as follows. In Section~\ref{seccf}, we review the groups $\text{SL}_2\left(\mathbb{Z}\right)$ and $\text{PSL}_2\left(\mathbb{Z}\right)$, and the connection between these groups and the theory of continued fractions. We also define the classical map $B_3\to \text{PSL}_2\left(\mathbb{Z}\right)$ that is central in this paper. In Section~\ref{seccfmap}, we establish an explicit linear time algorithm to solve the word problem in the braid group $B_3$. In particular, we associate finite continued fractions to elements of $B_3$ and we show that the continued fractions associated to a braid uniquely determine the braid. We also establish a new normal form for braids in $B_3$. Finally, in Section~\ref{seccpB_3}, we establish an explicit linear time algorithm to solve the conjugacy problem in $B_3$. In particular, we associate numerical class invariants to elements of $B_3$ that include (infinite) purely periodic continued fractions, and we show that the class invariants associated to a braid uniquely determine the conjugacy class of the braid. In the process, we also review the classical theory of conjugacy classes in $\text{SL}_2\left(\mathbb{Z}\right)$ and $\text{PSL}_2\left(\mathbb{Z}\right)$, and its connection to (infinite) eventually periodic continued fractions. We also give a new proof that $B_3$ is conjugacy separable.

\subsection{Acknowledgement}
I would like to express my sincere gratitude to Peter Ozsv{\'a}th for his encouragement during the completion of this work, and his many helpful comments on a previous version of this paper that resulted in a much improved exposition.

\section{The arithmetic of continued fractions and group theory}
\label{seccf}
In this section, we summarize well-known concepts, notation and terminology that we will use throughout the paper. In Subsection~\ref{subsecarithmetic}, we review finite continued fraction expansions of rational numbers and (infinite) periodic continued fraction expansions of real quadratic irrational numbers. In Subsection~\ref{subsecSL_2PSL_2}, we review the groups $\text{SL}_2\left(\mathbb{Z}\right)$ and $\text{PSL}_2\left(\mathbb{Z}\right)$ and their finite presentations, in terms of continued fraction expansions. In Subsection~\ref{subsecfses}, we review the braid group $B_3$ and its classical interpretation as a central extension of $\text{PSL}_2\left(\mathbb{Z}\right)$.

\subsection{Continued fractions}
\label{subsecarithmetic}

We denote the set of integers by $\mathbb{Z}$, we denote the set of rational numbers by $\mathbb{Q}$, and we denote the set of real numbers by $\mathbb{R}$. If $r\in \mathbb{Q}$, then a \textit{continued fraction expansion} of $r$ is an expansion of the form \[r= \left[c_0,c_1,\dots,c_n\right] = c_0 + \cfrac{1}{c_1 + \cfrac{1}{c_2 + \cfrac{1}{ \ddots + \cfrac{1}{c_n} }}}\] where $c_0\in \mathbb{Z}$ and $c_i\in \mathbb{Z}\setminus \{0\}$ for $1\leq i\leq n$. A consequence of the Euclidean algorithm is that every rational number admits a (non-unique) continued fraction expansion. However, every $r\in\mathbb{Q}$ admits a unique continued fraction expansion $\left[c_0,\dots,c_n\right]$ with $c_i>0$ for $1\leq i\leq n$. We refer to this as the \textit{simple continued fraction expansion} of $r$. 

An \textit{infinite continued fraction expansion} of a real number $r\in \mathbb{R}$ is an expansion of the form \[r= \left[c_0,c_1,\dots\right] = c_0 + \cfrac{1}{c_1 + \cfrac{1}{c_2 + \cfrac{1}{ \ddots + }}},\] with $c_i>0$ for $i>0$. We formally define the infinite continued fraction expansion as a limit $r = \lim_{n\to \infty} \left[c_0,c_1,\dots,c_n\right]$. The sequence in the limit is a Cauchy sequence and therefore the limit is well-defined. We remark that every irrational number has a unique infinite continued fraction expansion. An infinite continued fraction expansion is \textit{periodic} if there are nonnegative integers $n'\geq n$ such that $c_{n'+i} = c_{n+i}$ for all nonnegative integers $i$. Euler and Lagrange proved that a real number has a periodic continued fraction expansion if and only if it is a quadratic irrational number (i.e., an irrational number that satisfies a quadratic equation with integer coefficients). We denote a periodic continued fraction expansion by $\left[c_0,\dots,c_{n-1},\overline{c_n,\dots,c_{n'-1}},\dots\right]$, where the vinculum repeats indefinitely.

Let $k$ be minimal such that $\left[c_n,\dots,c_{n+k}\right]$ is a repeating portion of the continued fraction expansion of a real quadratic irrational number $r$. The set $P$ of minimal repeating portions of the periodic continued fraction expansion of $r$ consists of all cyclic shifts of $\left[c_n,\dots,c_{n+k}\right]$, and we can consider the lexicographical ordering on $P$. We refer to the \textit{fundamental period} of $r$ as the smallest element of $P$ with respect to the lexicographical ordering.

\subsection{The groups $\text{SL}_2\left(\mathbb{Z}\right)$ and $\text{PSL}_2\left(\mathbb{Z}\right)$}
\label{subsecSL_2PSL_2}

We recall the classical connection between continued fraction expansions and the groups $\text{SL}_2\left(\mathbb{Z}\right)$ and $\text{PSL}_2\left(\mathbb{Z}\right)$. In this paper, we will use this connection to study the braid group $B_3$ in terms of continued fraction expansions, by considering $B_3$ as a central extension of $\text{PSL}_2\left(\mathbb{Z}\right)$. The material in this subsection is well-known. We refer the reader to Chapter VII of~\cite{serrearithmetic} for a classical exposition of these groups from a number-theoretic perspective.

We recall that \[\text{SL}_2\left(\mathbb{Z}\right) = \left\{\begin{bmatrix} a & b \\ c & d \end{bmatrix}: a,b,c,d\in \mathbb{Z}\text{ and } ad-bc = 1\right\}\] is the group of $2\times 2$ integer matrices with determinant $+1$. The center of $\text{SL}_2\left(\mathbb{Z}\right)$ is the cyclic group $\left\langle -I \right\rangle$ of order $2$ generated by $-I$, where $I$ denotes the $2\times 2$ identity matrix in $\text{SL}_2\left(\mathbb{Z}\right)$. We recall that \[\text{PSL}_2\left(\mathbb{Z}\right) = \text{SL}_2\left(\mathbb{Z}\right)/\left\langle - I \right\rangle.\] 

We denote the affine plane over the field of rational numbers by $\mathbb{A}^2$ and the projective line over the field of rational numbers by $\mathbb{P}^1$. We use the notation $\left[x:y\right]$ to denote homogenous coordinates on $\mathbb{P}^1$. We can naturally identify $\mathbb{P}^1$ with the set of extended rational numbers $\mathbb{Q}\cup \{\infty\}$ by the bijective function $\mathbb{P}^1\to \mathbb{Q}\cup \{\infty\}$ defined in homogenous coordinates by the rule $\left[x:y\right]\to \frac{x}{y}$. We refer to the image of an element of $\mathbb{P}^1$ in $\mathbb{Q}\cup \{\infty\}$ as its affine coordinate. We recall that $\text{PSL}_2\left(\mathbb{Z}\right)$ can also be viewed as the group of integral linear fractional transformations $\mathbb{P}^1\to \mathbb{P}^1$ given in the affine coordinate by $z\to  \frac{az+b}{cz+d}$, where $a,b,c,d\in \mathbb{Z}$. 

The group $\text{SL}_2\left(\mathbb{Z}\right)$ acts on $\mathbb{A}^2$ in a natural manner and this action descends to an action of $\text{PSL}_2\left(\mathbb{Z}\right)$ on $\mathbb{P}^1$. A continued fraction expansion $\left[c_0,\dots,c_n\right]$ induces a function $\mathbb{P}^1\to \mathbb{P}^1$ defined by the rule $z\to \left[c_0,\dots,c_n+z\right]$ (if $z\in\mathbb{Q}$, then $c_n+z$ is not necessarily an integer, but we can still define the continued fraction expansion in the usual manner). We abuse notation and use the continued fraction expansion $\left[c_0,\dots,c_n\right]$ to refer to either this function, or a rational number (the special value of this function at $0$), depending on the context. The functions $\mathbb{P}^1\to \mathbb{P}^1$ obtained in the manner correspond to the elements of $\text{PSL}_2\left(\mathbb{Z}\right)$, and we will now recall their description in terms of a classical generating set for $\text{PSL}_2\left(\mathbb{Z}\right)$.

Let \[S = \begin{bmatrix} 1 & 1 \\ 0 & 1 \end{bmatrix},\text{ } T = \begin{bmatrix} 1 & 0 \\ -1 & 1 \end{bmatrix}.\] We view $S$ and $T$ as elements of $\text{SL}_2\left(\mathbb{Z}\right)$ or $\text{PSL}_2\left(\mathbb{Z}\right)$, depending on context. The finite generation of $\text{SL}_2\left(\mathbb{Z}\right)$ by $S$ and $T$ (and, consequently, the finite generation of the quotient group $\text{PSL}_2\left(\mathbb{Z}\right)$ by $S$ and $T$) is a consequence of the following proposition. We adopt standard arithmetic in the extended real numbers, in particular, the convention that $\frac{1}{0} = \infty$ and $\frac{1}{\infty} = 0$. 

\begin{proposition}
\label{propSL_2fg}
If \[\left(S^{a_1}T^{b_1}\cdots S^{a_n}\right)\begin{bmatrix} x \\ y \end{bmatrix} = \begin{bmatrix} x' \\ y' \end{bmatrix},\] then $\frac{x'}{y'} = \left[a_1,-b_1,\dots,a_{n-1},-b_{n-1},a_n+\frac{x}{y}\right]$. In particular, the groups $\text{SL}_2\left(\mathbb{Z}\right)$ and $\text{PSL}_2\left(\mathbb{Z}\right)$ are generated by $S$ and $T$, and the elements of $\text{PSL}_2\left(\mathbb{Z}\right)$ correspond to functions induced by continued fraction expansions.
\end{proposition}
\begin{proof}
The first statement is classical and follows by induction on $n$. We briefly recall why this statement implies that $\text{SL}_2\left(\mathbb{Z}\right)$ is generated by $S$ and $T$, to highlight how group-theoretic properties of $\text{SL}_2\left(\mathbb{Z}\right)$ are derived from properties of its faithful action on $\mathbb{A}^2$. The idea is that the stabilizer \[\text{Stab}\left(\begin{bmatrix} 0 \\ 1 \end{bmatrix}\right) = \left\langle T \right\rangle\] is the subgroup generated by $T$, and in particular the following equality of orbits \[\text{SL}_2\left(\mathbb{Z}\right)\begin{bmatrix} 0 \\ 1 \end{bmatrix} = \left\langle S, T \right\rangle \begin{bmatrix} 0 \\ 1 \end{bmatrix}\] would imply that $\text{SL}_2\left(\mathbb{Z}\right) = \left\langle S,T \right\rangle$. However, the equality of orbits follows from the first statement and the existence of a finite continued fraction expansion of every rational number. 
\end{proof}

We now improve Proposition~\ref{propSL_2fg} to establish a normal form for words in $S$ and $T$ using the existence and uniqueness of a simple continued fraction expansion of every rational number. Let $\iota:\mathbb{P}^1\to \mathbb{P}^1$ be the involution $\iota\left(z\right) = -\frac{1}{z}$, which corresponds to the element $STS\in \text{PSL}_2\left(\mathbb{Z}\right)$. 

\begin{proposition}
\label{propnorforPSL_2}
We can write each element of $\text{PSL}_2\left(\mathbb{Z}\right)$ uniquely as a function composition $\phi\circ \left[c_0,\dots,c_k\right]$, where $\phi\in \{I,\iota\}$, $\left[c_0,\dots,c_{k-1}\right]$ is the function induced by a simple continued fraction expansion of a positive rational number, and $c_k\in \mathbb{Z}$.
\end{proposition}
\begin{proof}
Note that $\mathbb{Q}\cap \left[0,\infty\right)$ is a fundamental domain for the action of the cyclic group $\{I,\iota\}$ on $\mathbb{Q}\cup \{\infty\}$, i.e., every orbit of the action contains a unique point of $\mathbb{Q}\cap \left[0,\infty\right)$. Proposition~\ref{propSL_2fg} implies that every element of $\text{PSL}_2\left(\mathbb{Z}\right)$ can be written uniquely in the form $\phi S^{a_1}T^{b_1}\cdots S^{a_{n-1}}T^{b_{n-1}}S^{a_n}T^{b_n}$, where $\phi\in \{I,STS\}$, $\left[a_1,-b_1,\dots,a_{n-1},-b_{n-1},a_n\right]$ is a simple continued fraction expansion with $a_1\geq 0$, and $b_n\in \mathbb{Z}$. Indeed, this corresponds to the existence and uniqueness of a simple continued fraction expansion with nonnegative first term of every nonnegative rational number. 
\end{proof}

Proposition~\ref{propnorforPSL_2} implies that the word problem in $\text{PSL}_2\left(\mathbb{Z}\right)$ with respect to the generating set $\{S,T\}$ admits a simple solution. In fact, the following finite presentation of $\text{PSL}_2\left(\mathbb{Z}\right)$ is well-known. 

\begin{proposition}
\label{propfpPSL_2}
We have an isomorphism of groups \[\text{PSL}_2\left(\mathbb{Z}\right) = \left\langle S,T : STS = TST,\text{ }\left(ST\right)^3 = I \right\rangle.\]
\end{proposition}

We remark that the relations in the presentation of $\text{PSL}_2\left(\mathbb{Z}\right)$ correspond to equalities of continued fraction expansions for all $z\in \mathbb{Q}\cup \{\infty\}$. For example, the relation $STS = TST$ corresponds to the equality $\left[1,-1,1,z\right] = \left[-1,1,-1,z\right]$ for all $z\in \mathbb{Q}\cup \{\infty\}$. Similarly, the relation $\left(ST\right)^3 = I$ corresponds to the equality $\left[1,-1,1,-1,1,-1,z\right] = z$ for all $z\in \mathbb{Q}\cup \{\infty\}$.

Proposition~\ref{propfpPSL_2} implies that $\text{PSL}_2\left(\mathbb{Z}\right)$ is isomorphic to the free product $\mathbb{Z}/2\mathbb{Z}\ast \mathbb{Z}/3\mathbb{Z}$. Indeed, $STS$ generates a copy of $\mathbb{Z}/2\mathbb{Z}$ in $\text{PSL}_2\left(\mathbb{Z}\right)$ and $ST$ generates a copy of $\mathbb{Z}/3\mathbb{Z}$ in $\text{PSL}_2\left(\mathbb{Z}\right)$. Furthermore, $\text{PSL}_2\left(\mathbb{Z}\right)$ is generated by $STS$ and $ST$ with these relations. Proposition~\ref{propfpPSL_2} also implies the classical statement that $\text{SL}_2\left(\mathbb{Z}\right)$ admits the finite presentation \[\text{SL}_2\left(\mathbb{Z}\right) = \left\langle S,T : STS = TST,\text{ }\left(ST\right)^6 = I \right\rangle.\] A traditional proof of these finite presentations is based on the ping-pong lemma, which gives a criterion for a group to be a free product of subgroups in terms of the actions of the group and its subgroups on a set (see II.B of~\cite{topicsggtpierre}).

\subsection{The fundamental short exact sequence}
\label{subsecfses}

We recall the algebraic definition of the braid group $B_3$ introduced by Artin. 

\begin{definition} 
\label{defB_3}
The braid group $B_3$ is defined by the Artin presentation: \[B_3 = \left\langle \sigma_1,\sigma_2:\sigma_1\sigma_2\sigma_1=\sigma_2\sigma_1\sigma_2\right\rangle.\] We refer to an element of $B_3$ as a braid, and we represent braids as words in the generators $\sigma_1$ and $\sigma_2$ (and their inverses). The Garside element of $B_3$ is $\Delta = \sigma_1\sigma_2\sigma_1$. 
\end{definition}

Let us define $\phi:B_3\to \text{PSL}_2\left(\mathbb{Z}\right)$ in terms of generators in the Artin presentation of $B_3$ by the rules $\phi\left(\sigma_1\right) = S$ and $\phi\left(\sigma_2\right) = T$. Proposition~\ref{propfpPSL_2} implies that $\phi$ is a well-defined surjective homomorphism.

\begin{proposition}
\label{propB_3PSL_2}
The sequence of maps \[1\to \left\langle \Delta^2 \right\rangle \to B_3\to \text{PSL}_2\left(\mathbb{Z}\right)\to 1\] is a short exact sequence. The center of $B_3$ is $\left\langle \Delta^2 \right\rangle$ and the short exact sequence above describes $B_3$ as a central extension of $\text{PSL}_2\left(\mathbb{Z}\right)$. 
\end{proposition}
\begin{proof}
The statement follows from Proposition~\ref{propfpPSL_2}, since $\phi\left(\Delta\right) = STS$. Indeed, $\Delta^2$ is central in $B_3$, and thus the subgroup of $B_3$ normally generated by $\Delta^2$ equals the subgroup of $B_3$ generated by $\Delta^2$. 
\end{proof}

In this paper, we refer to the short exact sequence in Proposition~\ref{propB_3PSL_2} as the \textit{fundamental short exact sequence}.

\section{The word problem in $B_3$}
\label{seccfmap}
In this section, we use the interpretation of the braid group $B_3$ as a central extension of the modular group $\text{PSL}_2\left(\mathbb{Z}\right)$ (Proposition~\ref{propB_3PSL_2}) to define an injective map $\rho:B_3\to \left(\mathbb{P}^1\right)^2\times \mathbb{Z}$, where $\mathbb{P}^1 = \mathbb{Q}\cup \{\infty\}$ can be identified with the set of extended rational numbers. If $g\in B_3$ is represented as a word in the Artin generators, then $\rho\left(g\right)$ can be defined in terms of the word using simple continued fractions. In particular, this results in an effective linear time algorithm to solve the word problem in $B_3$.

We begin by defining the map $\rho:B_3\to \left(\mathbb{P}^1\right)^2\times \mathbb{Z}$ based on the map $\phi:B_3\to \text{PSL}_2\left(\mathbb{Z}\right)$ in Proposition~\ref{propB_3PSL_2}.

\begin{definition}
If $g\in B_3$, then write $g = \prod_{i=1}^{n} \sigma_1^{a_i}\sigma_2^{b_i}$ as a word in the Artin generators. We define $\rho_1\left(g\right) = \left[a_1,-b_1,\dots,a_{n-1},-b_{n-1},a_n,-b_n\right]\in \mathbb{P}^1$ and $\rho_2\left(g\right)=\left[a_1,-b_1,\dots,a_n\right]\in \mathbb{P}^1$. We define the exponent of $g$ to be $\epsilon\left(g\right) = \sum_{i=1}^{n} \left(a_i+b_i\right)$. The map $\rho:B_3\to \left(\mathbb{P}^1\right)^2\times \mathbb{Z}$ is defined by the rule $\rho\left(g\right) = \left(\rho_1\left(g\right),\rho_2\left(g\right),\epsilon\left(g\right)\right)$.
\end{definition}

We can view $\rho_1\left(g\right)$ and $\rho_2\left(g\right)$ as special values $\left[a_1,-b_1,\dots,a_n,-b_n\right]\left(\infty\right)$ and $\left[a_1,-b_1,\dots,a_n,-b_n\right]\left(0\right)$, respectively, of the element $\phi\left(g\right)\in \text{PSL}_2\left(\mathbb{Z}\right)$ induced by the continued fraction expansion $\left[a_1,-b_1,\dots,a_n,-b_n\right]$ (see Subsection~\ref{subsecSL_2PSL_2}). We can view $\epsilon\left(g\right)$ as the image of $g$ under the abelianization homomorphism $\epsilon:B_3\to \mathbb{Z}$, where $\epsilon\left(\sigma_1\right) = 1 = \epsilon\left(\sigma_2\right)$. The interpretations of these extended rational numbers associated to a word representing $g$ show that they are independent of the choice of such a word, and in particular we have the following statement.

\begin{lemma}
\label{lemmainjB_3}
The map $\rho:B_3\to \left(\mathbb{P}^1\right)^2\times \mathbb{Z}$ defined by the rule $\rho\left(g\right) = \left(\rho_1\left(g\right),\rho_2\left(g\right),\epsilon\left(g\right)\right)$ is a well-defined injective map.
\end{lemma}
\begin{proof}
The statement that $\rho$ is well-defined follows from the discussion above. The statement that $\rho$ is injective follows from Proposition~\ref{propB_3PSL_2} since $\epsilon:B_3\to \mathbb{Z}$ restricted to the subgroup $\left\langle \Delta^2 \right\rangle$ is injective.
\end{proof}

Lemma~\ref{lemmainjB_3} is important in the sense that it allows us to determine a braid by a triple of numbers, each of which can be effectively computed. In particular, we use the map $\rho$ to give a fast solution to the word problem in the braid group $B_3$ with respect to the Artin presentation.

\begin{theorem}
The word problem in the braid group $B_3$ with respect to the Artin presentation can be solved in linear time. Indeed, $\rho\left(w\right)$ can be computed in running time $\mathcal{O}\left(L\right)$ if $w$ is a word in $B_3$ of length $L$, and a pair of words $w,w'$ in $B_3$ are equal if and only if $\rho\left(w\right) = \rho\left(w'\right)$. 
\end{theorem}
\begin{proof}
The map $\rho:B_3\to \left(\mathbb{P}^1\right)^2\times \mathbb{Z}$ is well-defined and injective according to Lemma~\ref{lemmainjB_3}. In particular, if $w$ and $w'$ are words in $B_3$, then $w = w'$ if and only if $\rho\left(w\right) = \rho\left(w'\right)$. The statement follows since each component of $\rho$ can be computed in running time $\mathcal{O}\left(L\right)$.
\end{proof}

We conclude this section with a new normal form for braids in $B_3$ based on the normal form for words in the generating set $\{S,T\}$ of $\text{PSL}_2\left(\mathbb{Z}\right)$ in Proposition~\ref{propnorforPSL_2}. The normal form is a variation of the Garside normal form~\cite{garside1969braid} in $B_3$.

\begin{theorem}
\label{thmnorfor}
If $g\in B_3$, then we can uniquely write $g = \Delta^{k}\prod_{i=1}^{n} \sigma_1^{a_i}\sigma_2^{b_i}$, where $a_1\geq 0$, $b_n\in \mathbb{Z}$, $a_i> 0$ for $2\leq i\leq n$, and $b_i<0$ for $1\leq i\leq n-1$.
\end{theorem}
\begin{proof}
Proposition~\ref{propnorforPSL_2} implies that $\phi\left(g\right) = \phi\left(\Delta^{\epsilon}\prod_{i=1}^{n} \sigma_1^{a_i}\sigma_2^{b_i}\right)$, where $\epsilon\in \{0,1\}$, and $\prod_{i=1}^{n} \sigma_1^{a_i}\sigma_2^{b_i}$ satisfies the desired conditions. The statement that $g$ admits an expression of the given form now follows from Proposition~\ref{propB_3PSL_2}. The uniqueness statement follows from the corresponding uniqueness statement in Proposition~\ref{propnorforPSL_2}.
\end{proof}

In the Garside normal form $g = \Delta^{k}\prod_{i=1}^{n} \sigma_1^{a_i}\sigma_2^{b_i}$, the exponents $a_i$ and $b_i$ of the Artin generators are nonnegative. On the other hand, in the normal form of Theorem~\ref{thmnorfor}, the exponents of $\sigma_1$ are nonnegative and the exponents of $\sigma_2$ are non-positive, except possibly the last exponent of $\sigma_2$.

\section{The conjugacy problem in $B_3$}
\label{seccpB_3}
In this section, we establish a linear time algorithm to solve the conjugacy problem in the braid group $B_3$. In Subsection~\ref{subseccpP}, we review the classical algorithm to solve the conjugacy problem in the groups $\text{SL}_2\left(\mathbb{Z}\right)$ and $\text{PSL}_2\left(\mathbb{Z}\right)$ in terms of continued fraction expansions. In Subsection~\ref{subseccpBP}, we use this algorithm and the surjective homomorphism $\phi:B_3\to \text{PSL}_2\left(\mathbb{Z}\right)$ to establish a linear time algorithm to solve the conjugacy problem in $B_3$. We also give a new proof that $B_3$ is conjugacy-separable.

\subsection{The conjugacy problem in $\text{PSL}_2\left(\mathbb{Z}\right)$}
\label{subseccpP}
In this subsection, we review the classical linear time algorithm to solve the conjugacy problem in the groups $\text{SL}_2\left(\mathbb{Z}\right)$ and $\text{PSL}_2\left(\mathbb{Z}\right)$ in terms of continued fraction expansions. A \textit{class function} on a group is a function that is constant on conjugacy classes. A class function is a \textit{complete class invariant} if it has different values on different conjugacy classes. We will define a class function on $\text{PSL}_2\left(\mathbb{Z}\right)$ and show that it is a complete class invariant.

In the sequel, we will identify elements of $\text{PSL}_2\left(\mathbb{Z}\right)$ (linear fractional transformations) with elements of $\text{SL}_2\left(\mathbb{Z}\right)$ ($2\times 2$ matrices) by thinking of a linear fractional transformation $z\to \frac{az+b}{cz+d}$ as either of the $2\times 2$ matrices $\pm \begin{bmatrix} a & b \\ c & d \end{bmatrix}$. 

We define the (absolute value of the) \textit{trace function} $\text{tr}:\text{PSL}_2\left(\mathbb{Z}\right)\to \mathbb{Z}$ by associating the number $\left|a+d\right|$ to the linear fractional transformation $z\to \frac{az+b}{cz+d}$. The trace function corresponds to the absolute value of the usual trace of matrices in $\text{SL}_2\left(\mathbb{Z}\right)$, and is therefore a class function on $\text{PSL}_2\left(\mathbb{Z}\right)$. However, the trace function is not a complete class invariant on $\text{PSL}_2\left(\mathbb{Z}\right)$. The conjugacy classes in $\text{PSL}_2\left(\mathbb{Z}\right)$ for each nonnegative value of the trace function are described as follows. Theorem~\ref{thmccPSL_2} states that this description is correct.

If $\text{tr}\left(A\right) = 0$, then $A$ is conjugate to the linear fractional transformation \[\pm \begin{bmatrix} 0 & 1 \\ -1 & 0 \end{bmatrix}.\] If $\text{tr}\left(A\right) = 1$, then $A$ is conjugate to precisely one of the two linear fractional transformations \[\text{ } \pm \begin{bmatrix} 0 & 1 \\ -1 & 1 \end{bmatrix},\text{ } \pm \begin{bmatrix} 0 & -1 \\ 1 & 1 \end{bmatrix}.\] The linear fractional transformations with trace less than two are referred to as \textit{elliptic transformations}. The elliptic transformations are the finite order linear fractional transformations and the action of an elliptic transformation on the complex projective line has a single fixed point. The trace zero elliptic transformations have order two and the trace one elliptic transformations have order three.

If $\text{tr}\left(A\right) = 2$, then $A$ is conjugate to a linear fractional transformation in the infinite family \[\pm \begin{bmatrix} 1 & s \\ 0 & 1 \end{bmatrix},\] for precisely one value of $s\in \mathbb{Z}$. The linear fractional transformations with trace equal to two are referred to as \textit{parabolic transformations} and the action of a parabolic transformation on the complex projective line has a single fixed point. 

If $\text{tr}\left(A\right)>2$, then $A$ is referred to as a \textit{hyperbolic transformation} and the action of $A$ on the real projective line has two fixed points corresponding to two linearly independent eigenvectors of the action of $A$ on the real affine plane. The fixed points of $A$ are real quadratic irrational numbers, and in particular, their continued fraction expansions are (eventually) periodic. We write $\xi\left(A\right)$ for the fundamental period of the (continued fraction expansion of the) fixed point corresponding to the eigenvalue of $A$ with absolute value greater than one. (The definition of the fundamental period of an infinite continued fraction is at the end of Subsection~\ref{subsecarithmetic}.) The conjugacy class of the hyperbolic transformation $A$ consists of all linear fractional transformations with the same trace and the same fundamental period.  

The statement that the conjugacy classes in $\text{PSL}_2\left(\mathbb{Z}\right)$ are as described directly above is well-known but we give a proof. We also refer the reader to~\cite{onishifraction} for a proof of the very similar description of the conjugacy classes in $\text{SL}_2\left(\mathbb{Z}\right)$. 

\begin{theorem}
\label{thmccPSL_2}
The conjugacy classes in $\text{PSL}_2\left(\mathbb{Z}\right)$ are as described directly above.
\end{theorem}
\begin{proof}
In this proof, we will identify a linear fractional transformation in $\text{PSL}_2\left(\mathbb{Z}\right)$ with its matrix representative in $\text{SL}_2\left(\mathbb{Z}\right)$ of nonnegative trace (or its pair of matrix representatives if both matrix representatives have trace zero). The fixed points of a linear fractional transformation in the complex projective line $\mathbb{P}^1$ can be identified with the eigenspaces of its matrix representatives in the complex affine plane $\mathbb{A}^2$.

Let us choose $A = \begin{bmatrix} a & b \\ c & d \end{bmatrix}$ such that $\left|a\right|$ is minimized in the conjugacy class of $A$. We claim that $\left|a\right|< \left|b\right|$ and $\left|a\right|< \left|c\right|$, unless $A = I_2$, the $2\times 2$ identity matrix. Let us assume, for a contradiction, that the claim is false. In this case, $a\neq 0$. If $\left|a\right|\geq\left|c\right|$, then conjugating $A$ by precisely one of the two matrices $S^{\pm 1} = \begin{bmatrix} 1 & \pm 1 \\ 0 & 1 \end{bmatrix}$ (depending on the signs of $a$ and $c$) will lower the value of $\left|a\right|$ by $\left|c\right|$. If $\left|a\right|\geq\left|b\right|$, then conjugating $A$ by precisely one of the two matrices $T^{\pm 1} = \begin{bmatrix} 1 & 0 \\ \mp 1 & 1 \end{bmatrix}$ (depending on the signs of $a$ and $b$) will lower the value of $\left|a\right|$ by $\left|b\right|$. We have a contradiction in either case by the minimality of $\left|a\right|$, unless $b = 0 = c$ and $A = I_2$. Therefore, the claim is established.

We will split the remainder of the proof into cases according to whether $A$ is an elliptic, parabolic, or hyperbolic transformation.

\begin{description}[style=unboxed,leftmargin=0cm]

\item[Case 1] ($0\leq \text{tr}\left(A\right)\leq 1$)

If $\text{tr}\left(A\right) = 0$, then we will prove that $A = \pm \begin{bmatrix} 0 & 1 \\ -1 & 0\end{bmatrix}$, and if $\text{tr}\left(A\right) = 1$, then we will prove that $A$ is either one of the two matrices $\begin{bmatrix} 0 & \pm 1 \\ \mp 1 & 1 \end{bmatrix}$.

Of course, $-bc = a^2 - \text{tr}\left(A\right)a + 1$ since $\text{det}\left(A\right) = 1$. If $a = 0$, then this implies that $bc = -1$. In this case, $A = \pm \begin{bmatrix} 0 & 1 \\ - 1 & 0 \end{bmatrix}$ if $\text{tr}\left(A\right) = 0$, and either $A = \begin{bmatrix} 0 & 1 \\ - 1 & 1 \end{bmatrix}$ or $A = \begin{bmatrix} 0 & -1 \\ 1 & 1 \end{bmatrix}$ if $\text{tr}\left(A\right) = 1$.

On the other hand, we will show that $a\neq 0$ is impossible. Indeed, we have that $b\neq 0$ and $c\neq 0$ are non-zero integers with the opposite sign since $0\leq \text{tr}\left(A\right)\leq 1$ and $\text{det}\left(A\right) = 1$. The hypothesis $\text{det}\left(A\right) = 1$ and the assumption $\left|a\right|<\left|b\right|,\left|c\right|$ imply that either $-bc  < -\left(b+1\right)\left(c-\text{tr}\left(A\right)\right) + 1$ if $b<0$ or $-bc < b\left(-c-1-\text{tr}\left(A\right)\right) + 1$ if $c<0$. In particular, if $b<0$, then $0<-c+\left(b+1\right)\text{tr}\left(A\right) + 1$, which is impossible since $c>0$. Similarly, if $c<0$, then $0<-b\left(1+\text{tr}\left(A\right)\right) + 1$, which is impossible since $b>0$. Therefore, we have considered all possibilities if $0\leq \text{tr}\left(A\right)\leq 1$, and the proof is complete in this case.

In fact, the constructive nature of the argument shows that if $A = \begin{bmatrix} a & b \\ c & d \end{bmatrix}$ with $\text{tr}\left(A\right) = 1$ and $\left|a\right|$ not necessarily minimized in the conjugacy class of $A$, then $A$ is conjugate to $\begin{bmatrix} 0 & 1 \\ - 1 & 1 \end{bmatrix}$ if $b>0$, and $A$ is conjugate to $\begin{bmatrix} 0 & -1 \\ 1 & 1 \end{bmatrix}$ if $c>0$. 

\item[Case 2] ($\text{tr}\left(A\right) = 2$)

If $\text{tr}\left(A\right) = 2$, then the characteristic polynomial of $A$ is $\left(x-1\right)^2$. The Cayley-Hamilton theorem implies that $\left(A-I\right)^2 = 0$. If $v\in \mathbb{Z}^2\setminus \{\left(0,0\right)\}$ is a non-zero vector, then either $w=Av - v$ is an eigenvector of $A$ with eigenvalue one, or $w = \left(0,0\right)$ in which case $v$ is an eigenvector of $A$ with eigenvalue one. If we rescale the eigenvector of $A$ with eigenvalue one, then we may assume that it is a primitive vector in $\mathbb{Z}^2$, and therefore it extends to a basis of $\mathbb{Z}^2$. The matrix of $A$ with respect to this basis is triangular with diagonal entries equal to one, since $\text{det}\left(A\right) = 1$. Therefore, $A$ is conjugate to $\begin{bmatrix} 1 & s \\ 0 & 1 \end{bmatrix}$ for some $s\in \mathbb{Z}$. The uniqueness of $s\in \mathbb{Z}$ follows from the fact that there is a unique primitive eigenvector of $A$ in $\mathbb{Z}^2$ with eigenvalue one, up to multiplication by $\pm 1$, unless $A=I$. 

In fact, the constructive nature of the argument shows that $s$ can be determined as follows, if we take $v = \left(1,0\right)$. If $a=1$ and $c = 0$, then $s = b$, since $v$ is an eigenvector of $A$ with eigenvalue one in this case. Otherwise, let $u$ be the greatest common divisor of $a-1$ and $c$, and choose $r,t\in \mathbb{Z}$ such that $\left(a-1\right)t - cr = u$. If $a\neq 1$, then $s = u\left(r + \frac{b}{a-1}t\right)$. If $c\neq 0$, then $s= u\left(r + \frac{d-1}{c}t\right)$.

\item[Case 3] ($\text{tr}\left(A\right) > 2$)

Finally, we consider a pair $A,B\in \text{PSL}_2\left(\mathbb{Z}\right)$ with $\text{tr}\left(A\right)=\text{tr}\left(B\right)>2$. In this case, the $2\times 2$ matrices representing $A$ and $B$ each have two distinct real eigenvalues with product equal to one, and the eigenspaces for these eigenvalues correspond to a pair of distinct fixed points of each of the corresponding linear fractional transformations. Let $\kappa_A$ and $\kappa_B$ be the fixed points corresponding to the eigenvalue with absolute value greater than one of the linear fractional transformations $A$ and $B$, respectively. 

Firstly, we claim that $A$ and $B$ are conjugate if and only if there is a linear fractional transformation $C$ such that $C\kappa_A = \kappa_B$. Indeed, if $A$ and $B$ are conjugate and $CAC^{-1} = B$, then $C$ maps the fixed points of $A$ to the fixed points of $B$ (with corresponding eigenvalues). Conversely, if there is a linear fractional transformation $C$ such that $C\kappa_A = \kappa_B$, then $CAC^{-1}$ and $B$ have a fixed point in common and thus must be the same linear fractional transformation.

Secondly, we claim that there is a linear fractional transformation $C$ such that $C\kappa_A = \kappa_B$ if and only if the fundamental periods of the continued fraction expansions of $\kappa_A$ and $\kappa_B$ are equal. (We refer to the end of Subsection~\ref{subsecarithmetic} for our definition of the fundamental period of a continued fraction expansion.) If $z\in \mathbb{R}$ is a quadratic irrational number and $C$ is a linear fractional transformation, then a classical theorem of Serret implies that the fundamental period of $Cz$ is the same as the fundamental period of $z$ (we refer to the appendix of~\cite{bombiericontinuedfraction} for a short computational proof). Therefore, if $C$ is a linear fractional transformation such that $C\kappa_A = \kappa_B$, then the fundamental periods of $\kappa_A$ and $\kappa_B$ are equal. Conversely, if the continued fraction expansion of $\kappa_A$ is $\left[c_0,\dots,c_{n-1},\overline{c_n,\dots,c_{n'-1}},\dots\right]$, and if $C'$ is the inverse of the linear fractional transformation induced by the continued fraction expansion $\left[c_0,\dots,c_{n-1}\right]$, then the continued fraction expansion of $C'\kappa_A$ is $\left[\overline{c_n,\dots,c_{n'-1}},\dots\right]$. In particular, if the fundamental periods of the continued fraction expansions of $\kappa_A$ and $\kappa_B$ are equal, then there is a linear fractional transformation $C$ such that $C\kappa_A=\kappa_B$.

Therefore, $A$ and $B$ are conjugate if and only if the fundamental periods of the continued fraction expansions of $\kappa_A$ and $\kappa_B$ are equal, i.e., if and only if $\xi\left(A\right) = \xi\left(B\right)$.
\end{description}
We have considered all cases and the statement is established.
\end{proof}

We remark that there are other solutions to the conjugacy problem in $\text{PSL}_2\left(\mathbb{Z}\right)$. For example, it is a classical fact that $\text{PSL}_2\left(\mathbb{Z}\right)$ is isomorphic to the free product of cyclic groups $\mathbb{Z}/2\mathbb{Z}\text{ }\ast\text{ }\mathbb{Z}/3\mathbb{Z}$ (see Proposition~\ref{propfpPSL_2} and the subsequent discussion), and the solution to the conjugacy problem in free products of cyclic groups is well-known~\cite{magnuscgt}. Indeed, if $w=\prod_{i=1}^{n} g_i^{e_i}$ is a reduced word in such a free product with $e_i\in \{\pm 1\}$ for $1\leq i\leq n$, then the \textit{cyclic shifts} of $w$ are $w_j = \left(\prod_{i=j}^{n} g_i^{e_i}\right)\left(\prod_{i=1}^{j-1} g_i^{e_i}\right)$ for $1\leq j\leq n$. Of course, $w_j$ is conjugate to $w$ for $1\leq j\leq n$. However, it is also true in this case that a pair of reduced words are conjugate if and only if one of the words is a cyclic shift of the other. The main observation is that distinct reduced words in free products of cyclic groups represent different group elements.

Let us now define a class function $\lambda:\text{PSL}_2\left(\mathbb{Z}\right)\to \mathbb{C}$, which we will establish is a complete class invariant using Theorem~\ref{thmccPSL_2}. Let $i=\sqrt{-1}$ be the imaginary unit and let $\omega = e^{2\pi i/3}$ be the cubic root of unity with positive imaginary part. We represent the linear fractional transformation $A\in \text{PSL}_2\left(\mathbb{Z}\right)$ by a matrix $\begin{bmatrix} a & b \\ c & d \end{bmatrix}\in \text{SL}_2\left(\mathbb{Z}\right)$ of nonnegative trace. If $\text{tr}\left(A\right) = 1$, then $b,c\neq 0$ have the opposite sign since $\text{det}\left(A\right) = 1$. If $\text{tr}\left(A\right) = 2$, then we define $s\left(A\right)\in \mathbb{Z}$ to be such that $A$ is conjugate to $\begin{bmatrix} 1 & s\left(A\right) \\ 0 & 1 \end{bmatrix}$. If $\text{tr}\left(A\right)>2$, then we define $\chi\left(A\right)$ to be the real quadratic irrational number with purely periodic continued fraction expansion $\left[\overline{\xi\left(A\right)},\dots\right]$. We now define $\lambda\left(A\right)$ as follows: \[\lambda\left(A\right) = \begin{cases} 
      i & \text{if } \text{tr}\left(A\right) = 0\\
     \omega &  \text{if } \text{tr}\left(A\right) = 1\text{ and } b>0\\
     -\omega & \text{if } \text{tr}\left(A\right) = 1\text{ and } c>0\\
     s\left(A\right) & \text{if } \text{tr}\left(A\right) = 2 \\
     \chi\left(A\right) & \text{if } \text{tr}\left(A\right) > 2
   \end{cases}.
\] If $\text{tr}\left(A\right)\neq 2$, then $\lambda\left(A\right)$ is a fixed point of a canonical representative of the conjugacy class of $A$. If $\text{tr}\left(A\right) = 2$, then $\lambda\left(A\right)$ is the translation distance of a canonical representative of the conjugacy class of $A$.

\begin{corollary}
\label{corcPSLlt}
The class function $\lambda:\text{PSL}_2\left(\mathbb{Z}\right)\to \mathbb{C}$ is a complete class invariant, i.e., $\lambda\left(A\right) = \lambda\left(B\right)$ for $A,B\in \text{PSL}_2\left(\mathbb{Z}\right)$ if and only if $A$ and $B$ are conjugate in $\text{PSL}_2\left(\mathbb{Z}\right)$. Furthermore, $\lambda\left(A\right)$ can be computed in running time $\mathcal{O}\left(L\right)$, where $L$ is the word length of $A$. In particular, the conjugacy problem in $\text{PSL}_2\left(\mathbb{Z}\right)$ can be solved in linear time.
\end{corollary}
\begin{proof}
The statement that $\lambda$ is a complete class invariant is a reformulation of Theorem~\ref{thmccPSL_2}. We show that $\lambda\left(A\right)$ can be computed in linear time as a function of the word length of $A$ in the generating set $\{S,T\}$ of $\text{PSL}_2\left(\mathbb{Z}\right)$. Of course, the entries of $A$ can be computed in linear time. In particular, the trace $\text{tr}\left(A\right)$ can be computed in linear time. Furthermore, $\lambda\left(A\right)$ can be computed in linear time if $\text{tr}\left(A\right)\leq 2$ by the definition of $\lambda\left(A\right)$ in this case. Finally, if $\text{tr}\left(A\right)>2$, then we will show that the fundamental period $\xi\left(A\right)$ can be computed in linear time as a function of the word length of $A$. 

Indeed, if $A = S^{a_1}T^{b_1}\dots S^{a_n}T^{b_n}$ with $a_1\geq 0$, $b_n\leq 0$, $a_i>0$ for $2\leq i\leq n$, and $b_i<0$ for $1\leq i\leq n-1$, then the infinite purely periodic continued fraction $\left[\overline{a_1,-b_1,\dots,a_n,-b_n},\dots\right]$ is a fixed point of $A$ if $a_1,b_n\neq 0$, and the infinite periodic continued fraction $\left[0,\overline{-b_1,a_2,\dots,-b_{n-1},a_n},\dots\right]$ is a fixed point of $A$ if $a_1 = 0 = b_n$. On the other hand, if $a_1 = 0$ and $b_n\neq 0$, then the infinite periodic continued fraction $\left[0,-b_1,\overline{a_2,-b_1,\dots,a_n,-b_n-b_1},\dots\right]$ is a fixed point of $A$, and if $a_1\neq 0$ and $b_n = 0$, then the infinite periodic continued fraction $\left[a_1,\overline{-b_1,a_2,\dots,-b_{n-1},a_n+a_1},\dots\right]$ is a fixed point of $A$. In all of these cases, the fundamental period $\xi\left(A\right)$ can be computed in linear time as a function of the word length of $A$. In general, if $A$ is expressed as a reduced word in the generating set $\{ST,STS\}$, then a conjugate of $A$ in the special form just discussed must be a cyclic shift of $A$, and a cyclic shift in this special form can be determined in linear time as a function of the word length of $A$. Indeed, we use the solution to the word problem in $\text{PSL}_2\left(\mathbb{Z}\right)$ outlined immediately after the proof of Theorem~\ref{thmccPSL_2}, and we observe that a word in the generating set $\{ST,STS\}$ is in the special form just discussed if and only if it begins with an $ST$ and ends with an $STS$. Therefore, the fundamental period $\xi\left(A\right)$ can be determined in linear time as a function of the word length of $A$ if $\text{tr}\left(A\right)>2$.
\end{proof}

\subsection{The relationship between conjugation in $B_3$ and conjugation in $\text{PSL}_2\left(\mathbb{Z}\right)$}
\label{subseccpBP}

In this subsection, we establish a linear time algorithm to solve the conjugacy problem in $B_3$. The algorithm is based on the linear time algorithm to solve the conjugacy problem in $\text{PSL}_2\left(\mathbb{Z}\right)$ in Subsection~\ref{subseccpP} and the fundamental short exact sequence \[1\to \left\langle \Delta^2 \right\rangle \to B_3\to \text{PSL}_2\left(\mathbb{Z}\right)\to 1\] in Proposition~\ref{propB_3PSL_2}. Let $\mathbb{C}\left(x\right)$ denote the field of rational functions in one variable over $\mathbb{C}$. Let $\mu:B_3\to \mathbb{C}\left(x\right)$ be defined by the rule $\mu\left(g\right) = \lambda\left(g\right)x^{\epsilon\left(g\right)}$ for $g\in B_3$, where $\lambda:\text{PSL}_2\left(\mathbb{Z}\right)\to \mathbb{C}$ is the complete class invariant introduced in Subsection~\ref{subseccpP} (see the discussion preceding Corollary~\ref{corcPSLlt}) and $\epsilon:B_3\to \mathbb{Z}$ is the abelianization homomorphism. 

\begin{theorem}
The class function $\mu:B_3\to \mathbb{C}\left(x\right)$ is a complete class invariant, i.e., $g,h\in B_3$ are conjugate in $B_3$ if and only if $\mu\left(g\right) = \mu\left(h\right)$. If $g\in B_3$, then $\mu\left(g\right)$ can be computed in running time $\mathcal{O}\left(L\right)$ if the word length of $g$ is $L$. In particular, the conjugacy problem in the braid group $B_3$ can be solved in linear time.
\end{theorem}
\begin{proof}
Let $g,h\in B_3$ and consider the surjective homomorphism $\phi:B_3\to \text{PSL}_2\left(\mathbb{Z}\right)$. Theorem~\ref{thmccPSL_2} implies that $\mu\left(g\right) = \mu\left(h\right)$ if and only if $\phi\left(g\right)$ and $\phi\left(h\right)$ are conjugate in $\text{PSL}_2\left(\mathbb{Z}\right)$ and $\epsilon\left(g\right) = \epsilon\left(h\right)$. However, $\phi\left(g\right)$ and $\phi\left(h\right)$ are conjugate in $\text{PSL}_2\left(\mathbb{Z}\right)$ if and only if $g$ is conjugate to $\Delta^{2k}h$ in $B_3$ for some $k\in \mathbb{Z}$ by the fundamental short exact sequence (Proposition~\ref{propB_3PSL_2}). Furthermore, if $g$ is conjugate to $\Delta^{2k}h$, then $k = \frac{\epsilon\left(g\right) - \epsilon\left(h\right)}{6}$ since $\epsilon:B_3\to\mathbb{Z}$ is a class function. In particular, $k=0$ if and only if $\epsilon\left(g\right) = \epsilon\left(h\right)$. Therefore, $g,h\in B_3$ are conjugate in $B_3$ if and only if $\mu\left(g\right) = \mu\left(h\right)$.

The statement that $\mu\left(g\right)$ can be computed in running time $\mathcal{O}\left(L\right)$ if the word length of $g$ is $L$, follows from the analogous statement concerning $\lambda\left(\phi\left(g\right)\right)$ in Corollary~\ref{corcPSLlt}, and the statement that $\epsilon\left(g\right)$ can be computed in running time $\mathcal{O}\left(L\right)$.
\end{proof}

Finally, we show that the braid group $B_3$ is conjugacy-separable. A group $G$ is \textit{conjugacy separable} if for each pair of non-conjugate elements $g,h\in G$, there is a finite index normal subgroup $N\subseteq G$ (depending on $g$ and $h$), such that the images of $g$ and $h$ in $G/N$ are non-conjugate in $G/N$. The conjugacy separability of the fundamental groups of compact oriented Seifert-fibered $3$-manifolds has already been established~\cite{allenbyconsep}. In particular, it is known that the braid group $B_3$ is conjugacy separable, since it is the fundamental group of the trefoil knot complement. However, we give a new and simpler proof that $B_3$ is conjugacy separable in the spirit of this paper.

\begin{theorem}
The braid group $B_3$ is conjugacy separable. 
\end{theorem}
\begin{proof}
Let us assume, for a contradiction, that $g,h\in B_3$ are non-conjugate elements but that for each finite index normal subgroup $N\subseteq B_3$, the images of $g$ and $h$ are conjugate in $G/N$. In particular, $\phi\left(g\right)$ and $\phi\left(h\right)$ are conjugate in $\text{PSL}_2\left(\mathbb{Z}\right)$ since the group $\text{PSL}_2\left(\mathbb{Z}\right)$ is conjugacy separable~\cite{stebeconsep}. The fundamental short exact sequence (Proposition~\ref{propB_3PSL_2}) implies that $g$ and $\Delta^{2k}h$ are conjugate in $B_3$ for some $k\in \mathbb{Z}$. 

However, if $\epsilon_N:B_3\to \mathbb{Z}/N\mathbb{Z}$ denotes the reduction of $\epsilon:B_3\to \mathbb{Z}$ modulo $N$, then $\epsilon_N\left(g\right) = \epsilon_N\left(h\right)$ for all $N\in \mathbb{Z}$ since $\mathbb{Z}/N\mathbb{Z}$ is abelian. In particular, $k = \frac{\epsilon\left(g\right) - \epsilon\left(h\right)}{6}$ is divisible by every positive integer, and we deduce that $k=0$. Therefore, $g$ and $h$ are conjugate in $B_3$ and this contradiction completes the proof. 
\end{proof}

\bibliography{bibresearchstatement}
\bibliographystyle{plain}
\end{document}